\documentclass[12pt, a4paper]{amsart}
\usepackage[utf8]{inputenc}
\usepackage{graphicx,graphics}
\usepackage{amsmath,amsthm,amssymb, amstext,amscd,latexsym,amsfonts}
\usepackage{hyperref}

\theoremstyle{plain}
\newtheorem{theorem}{Theorem}[section]
\newtheorem{lemma}[theorem]{Lemma}
\newtheorem{corollary}[theorem]{Corollary}
\newtheorem{proposition}[theorem]{Proposition}
\newtheorem{claim}[theorem]{Claim}

\theoremstyle{definition}
\newtheorem{definition}[theorem]{Definition}

\newtheorem{remark}[theorem]{Remark}

\newtheorem{notation}[theorem]{Notation}

\newcommand{\R}{{\mathbb R}}
\newcommand{\C}{{\mathbb C}}

\newcommand{\F}{{\mathbb F}}

\newcommand{\Ker}{{\operatorname{Ker}}}

\newcommand{\Sc}{{\mathcal S}}

\newcommand{\tG}{{\widetilde{G}}}
\newcommand{\tC}{{\widetilde{C}}}
\newcommand{\tO}{{\widetilde{O}}}
\newcommand{\Fou}{{\mathcal{F}}}
\renewcommand{\ss}{{\subseteq}}

\title [Multiplicity one theorem for a positive characteristic] {Multiplicity one theorem for $(\mathrm{GL}_{n+1},\mathrm{GL}_n)$ over a local field of positive characteristic}
\author{Dor Mezer}
\date{\today}

 
\begin{document}

\begin{abstract}
Let $\mathbb{F}$ be a non-archimedean local field of positive characteristic
different from 2. We consider distributions on $\mathrm{GL}(n+1,\mathbb{F})$
which are invariant under the adjoint action of $\mathrm{GL}(n,\mathbb{F})$. We prove that any such distribution is invariant with respect to transposition. This implies that the
restriction to $\mathrm{GL}(n,\F)$ of any irreducible smooth representation of $\mathrm{GL}(n+1,\F)$ is multiplicity free.
\end{abstract}

\maketitle
\tableofcontents
\section {Introduction}
Let $\F$ be a non-archimedean local field of positive characteristic different from 2.  
Consider the standard embedding of $\mathrm{GL}(n,\F)$ into $\mathrm{GL}(n+1,\F)$, and let $\mathrm{GL}(n,\F)$ act on $\mathrm{GL}(n+1,\F)$ by conjugation.
The goal of this paper is to prove the following theorem.
\begin {theorem} \label{goal}
Every $\mathrm{GL}(n,\F)$-invariant distributution on $\mathrm{GL}(n+1,\F)$ is invariant with respect to transposition.
\end{theorem}

Using the Gelfand-Kazhdan criterion (see \cite{GK}), this theorem implies the following one. For more details, see section 1 of \cite{AGRS}.
\begin{theorem} \label{mult1}
Let $\pi$ be an irreducible smooth representation of $\mathrm{GL}(n+1,\F)$, and $\rho$ be an irreducible smooth representation of $\mathrm{GL}(n,\F)$. Then
$$\dim\mathrm{Hom}_{\mathrm{GL}(n,\F)}(\pi|_{\mathrm{GL}(n,\F)},\rho)\leq 1$$
\end{theorem}
For theorems and conjectures on when the dimension is 1 and when it is zero see \cite{GGPW}.\\
An important direct consequence of theorem \ref{goal} is the following:
\begin{theorem} \label{Kir}
Let $P_n$ be the subgroup of $\mathrm{GL}(n,\F)$ consisting of matrices whose last row is $(0,\dots,0,1)$. Then any distribution on $\mathrm{GL}_n$ invariant to conjugation by $P_n$ is also invariant to conjugation by $\mathrm{GL}_n$.
\end{theorem}
This theorem can be used to prove Kirillov's conjecture, as well as the following theorem (see \cite{Ber} for both).
\begin{corollary}
the Bernstein-Zelevinsky product of two irreducible unitary representations of $\mathrm{GL}(n,\F)$ and $\mathrm{GL}(m,\F)$ is an irreducible unitary representation of $\mathrm{GL}(n+m,\F)$.
\end{corollary}

\emph{Remarks.}
\begin{enumerate}
\item Theorems \ref{mult1} and \ref{Kir} were already known using other methods (see \cite{AAG} for Theorem \ref{mult1} and \cite{Ber} for \ref{Kir}).

\item The starting point of this work was a discussion between my advisor and Guy Henniart in Summer 2017.

\item This whole paper is heavily influenced by the proof of an analog to Theorem \ref{goal}, for local fields of characteristic 0. This is given in \cite{AGRS}. We follow the exposition in \cite{0-char} (one may compare this paper to \cite{0-char} and find the many similarities and identical parts).\\
\end{enumerate}

\subsection{Sketch of Proof}

We prove our theorem by induction on $n$.

We use the notations $V=\F^n$, $\tG=\mathrm{GL}(V)\rtimes S_2$, where the non-trivial element of $S_2$ acts on $\mathrm{GL}(V)$ by $g\mapsto \ ^tg^{-1}$. Let $\chi$ be the pullback to $\tG$ of the sign character on $S_2$.
We shall denote $gl(V)\times V\times V^*$ by $X$, and by $\Delta:X\to \F[x]$ we will denote the map $(A,v,\phi)\mapsto \mathrm{ch}A$, where $\mathrm{ch}$ denotes the characteristic polynomial map.\\
We think of $X$ as sitting inside $\mathrm{gl}(n+1,\F)$ by matrices whose $n+1,n+1$ entry is 0. This way, we get an action of $G$ on $X$ by conjugation. If we let the generator of $S_2$ act on $X$ by transposition, we get a consistent action of $\tG$. One can easily see that this action is a product of an action on $gl(V)$ and an action on $V\times V^*$.\\
Throughout the proof we will use two powerful tools handling distributions, which we call 'the Localization principle' and 'Frobenius descent'. They are described in Section \ref{Preliminaries}. Morally, what they allow us is to treat distributions in a more geometric way, similar to the way in which one treats functions.\\

The first step of the proof is a reformulation of the theorem \ref{goal}, as the statement that any $(\tG,\chi)$-equivariant distributions on $X$ is 0.
\\The main strategy is to restrict the possible support of a $(\tG,\chi)$-equivariant distributions on $X$. Suppose throughout this sketch of proof that $(A,v,\phi)\in X$ is a point in the support of such a distribution.\\


Denote by $\phi v$ the pairing between $\phi$ and $v$. Using Frobenius descent and the Localization principle, we show that necessarily \begin{equation}\label{sketchperp}
    \phi v=0
\end{equation}
\\Next we introduce automorphisms of the problem, which will move the support, and so will restrict the intersection of all possible supports even further:
\begin{itemize}
    \item For $\lambda \in \F$, let $\nu_\lambda:X\to X,$ be the homeomorphism defined by $$\nu_\lambda (A,v,\phi)=(A+\lambda v\otimes \phi,v,\phi)$$
    \item Let $f\in \F (x)^\times$, and fix a fiber $F$ of $\Delta$, at a characteristic polynomial coprime to $f$ (both to the numerator and to the denominator).
    Let $\rho _f:F\to F$ be the homeomorphism defined by $$\rho _f (A,v,\phi)=(A,f(A)v,\phi f(A))$$
\end{itemize}
By $\phi f(A)$ we mean $f(A)^* \phi$, where $f(A)$ is the dual operator to on $V^*$ to $f(A)$.\\
Note that we may think of the second one as an automorphism of the problem using the Localization principle.
\\
\\Since these automorphisms must keep $A$ inside the union of the possible supports of all $(\tG,\chi)$-equivariant distributions, we can apply (\ref{sketchperp}) to $\rho_f (A,v,\phi)$ in order to get $\phi f(A)^2 v=0$. Since this is true for a dense set of polynomials $f$, it is true for all polynomials, and so for any polynomial we also have
$$\phi f(A) v= \frac{\phi(1+f(A))^2v-\phi(f(A))^2v-\phi v}{2}=0$$
\\By a theorem in linear algebra that we shall prove (Theorem \ref{LinAlg} below), this last condition is equivalent to the fact that $\nu_\lambda$ keeps $(A,v,\phi)$ inside the same fiber of $\Delta$. Denote by $R$ the subset of $X$ satisfying this condition.
\\
\\Localizing to the fiber of some $g\in\F[x]$ with respect to $\Delta$, we can use a method of stratification. Denote by $P_i$ the union of all $\tG$-orbits of dimension at most $i$ of matrices with characteristic polynomial $g$. Let $R_i:=(P_i\times V \times V^*)\cap R$. For any open orbit $O$ of $P_i$ set
$$\tO:=(O\times V \times V^*)\cap\bigcap_{\lambda\in\F}\nu_\lambda^{-1}(R_i)$$
\\Matrices with characteristic polynomial $g$ consist of finitely many orbits, and so our strategy will be to show by downward induction that $A\in P_i$.
\\The induction basis is clear (for large enough $i$), because of finiteness. For the induction step, it is enough to restrict to one of the open $\tG$ orbits in $P_i$, say $O$, and show that the only $(\tG,\chi)$-equivariant distribution on $\tO$ is $0$.

At this point, we bring in Fourier transform. We use $\Fou_{V \oplus V^*}$ to denote the Fourier transform on the $V\oplus V^*$ coordinates with respect to the bilinear form induced by the quadratic form $Q((v,\phi)):=\langle \phi , v\rangle
:=\phi(v)$.

We formulate a sufficient condition for the induction step, and from now on we will focus on proving it:

\begin{claim} \label{sk_orbit}
Let $O$ be an open $\tG$ orbit of $P_i$. Suppose $\xi$ is a $(\tG,\chi)$-equivariant distribution on $\Delta^{-1}(g)$ such that
$$\mathrm{supp}(\xi) \ss \tO$$
and
$$\mathrm{supp}(\Fou_{V \oplus V^*}(\xi)) \ss \tO$$
Then $\xi=0$.
\end{claim}

We shall use the helpful notations $$Q_A:=\{(v,\phi)\in V\oplus V^*|v\otimes\phi\in [A,\mathrm {gl} (V)]\}$$ and $$R_A:=\{(v,\phi)\in V\oplus V^*|\forall k\geq 0,\ \phi A^k v=0\}=\{(v,\phi)\in V\oplus V^*|(A,v,\phi)\in R\}$$
We show that $Q_A\ss R_A$, and that if $(A,v,\phi)\in \tO$, as we assume it to be, then $(v,\phi)\in Q_A$.
\\We prove $Q_{A_1}\oplus Q_{A_2}\ss Q_{A_1\oplus A_2}$, which allows us (using the Localization principle and Frobenius descent) to reduce claim \ref{sk_orbit} to the case that $A$ is a companion matrix (See definition \ref{comp} below), and show that there are no distributions on $R_A$ (and so in particular on $Q_A)$ which are equivariant with respect to the centralizer of $A$ inside $\tG$ (call it $\widetilde{C}_A$) with character $\chi$.
\\
\\For this, assume $\mathrm{ch}A=f(x)^s$, where $f\in F[x]$ is irreducible. We introduce the descending filtration $U_i:=f(A)^iV$, and the dual descending filtration $U_i^*:=U_{s-i}^\perp=f(A^*)^i V^*$ on $V^*$.
\\We prove that $R_A=\bigcup_{i=0}^s U_i\oplus U_{s-i}^*$, and this will allow us to prove our claim using a theorem of Rallis and Schiffmann (Theorem \ref{Metaplectic} below).
\\This theorem states that given a distribution $\xi$ on a vector space $W$ with a quadratic form $Q$, such that both the support of $\xi$ and the support of $\Fou_Q\xi$ are inside the zeros of $Q$, then $\xi$ is 'abs-homogeneous' of degree $\frac{1}{2}\dim W$ (see Definition \ref{abs-hom} below).
\\We use this theorem for the restriction of our distribution to $U_0\oplus U_s$, where $(\tC_A,\chi)$-equivariance implies 'abs-homogeneity' of degree $0$ (which is just invaraince to homothety), and so we know this restriction is 0.
This method will allow us to reduce to the statement for smaller $s$ and finish by induction.

\subsection{Related Results}

The result of this paper (along with the discussed consequences) is already known for non-archimedean local fields of characteristic $0$ (see \cite{AGRS}) and for the fields $\R$ and $\C$ (see \cite{AG-R,SZ}).\\
For finite fields, however, Theorems \ref{goal} and \ref{mult1} are not true.
There is a weaker result than Theorem \ref{mult1} that is known for all local and finite fields, and is shown in \cite{AG-Jacquet}.
Another weaker result was shown for all local fields of arbitrary characteristic in \cite{AGS}.\\
A possible direction to continue the work of this paper would be to prove an analog result to theorem \ref{goal} for orthogonal and unitary groups. It is more than likely that proofs of such statements would rely on this result for $\mathrm{GL}(n)$, as this was the case with local fields of $0$ characteristic (see \cite{AGRS,SZ}).\\
Theorem \ref{goal} is conjectured to be true also if $\F$ is a non-archimedean local field of characteristic 2, and it seems that this case requires some more thought.

\subsection{Acknowledgements}
I would like to deeply thank my advisor, {\bf Dmitry Gourevitch}, for guiding me through this project, for exposing me to this fascinating area of mathematics, and for teaching me the required background in mathematics. I would also like to thank him for the exceptional availability and willingness to help throughout the process.\\
I deeply thank {\bf Guy Henniart} for sharing his work on this problem with us.
\\I would also like to thank {\bf Avraham Aizenbud} for his help along the way.
\\I want to thank my friends with whom I have discussed this work and who have given me helpful feedback, among which are {\bf Shachar Carmeli}, {\bf Guy Kapon}, and {\bf Guy Shtotland}.
\\In addition, I would like to thank {\bf Lev Radzivilovsky}, {\bf Shachar Carmeli}, and {\bf Guy Kapon} for teaching me a huge amount of mathematics in the past and in the present.

D.M. was partially supported by ERC StG grant 637912.
\section {Preliminaries} \label{Preliminaries}

We will use the standard terminology of $l$-spaces introduced in
\cite{BZ}, section 1. We denote by $\Sc(Z)$ the space of Schwartz
functions on an $l$-space $Z$, and by $\Sc^*(Z)$ the space of
distributions on $Z$ equipped with the weak topology.

\begin {notation} 
When we have a vector space $V$, we denote $\mathrm{End}(V)=\mathrm{Hom}(V,V)$ by $gl(V)$. We also use $\mathrm{gl}_n$ and $\mathrm{gl}(n,\F)$ in the same sense. Given $v\in V,\phi\in V^*$, we use $\phi v$ to denote the pairing usually denoted by $\langle\phi,v\rangle$ or by $\phi(v)$. Similarly, if we also have $A\in \mathrm{gl}(V)$, we use the notation $\phi A v=\langle\phi,Av\rangle$. We also use the notation $\phi A=A^*\phi$. These notations are consistent with matrix multiplication.\\
Another notation we use is $v\otimes \phi\in \mathrm{gl}(V)$, which is defined by $(v\otimes \phi) (u):=\phi(u)v$.
\end{notation}

%
\begin{notation} [Fourier transform]
Let $W$ be a finite dimensional vector space over $\F$. Let  $B$ be
a nondegenerate symmetric bilinear form on $W$. We denote by
$\Fou_B:\Sc^*(W) \to \Sc^*(W)$ the Fourier transform defined using
$B$ and the self-dual measure on $W$.

By abuse of notation, we also denote by $\Fou_B$ the partial Fourier
transform $\Fou_B:\Sc^*(Z \times W) \to \Sc^*(Z\times W)$ for any
$l$-space $Z$.

If $W=U \oplus U^*$ then it has a canonical symmetric bilinear form
given by the quadratic form $Q((v,\phi)):=\phi v$. We will denote the Fourier transform defined by it
simply by $\Fou_W$. If $W$ is clear from the context, we sometimes ommit it from the notation and denote  $\Fou=\Fou_W$.
\end{notation}
\begin{proposition} \label{2Four}
Let $W_1 \oplus W_2$ be finite dimensional vector spaces. Let $B_1$
and $B_2$ be nondegenerate symmetric bilinear forms on $W_1$ and
$W_2$ respectively. Let $Z \subset W_1$ be a closed subset. Let $\xi
\in \Sc^*(W_1 \oplus W_2)$ be a distribution. Suppose that
$\Fou_{B_1 \oplus B_2}(\xi)$ is supported in $Z\times W_2$. Then
$\Fou_{B_1}(\xi)$ is also supported in $Z\times W_2$.
\end{proposition}
\begin{proof}
Let $p_1$ denote the projection $W_1\oplus W_2 \to
W_1$. Since $\Fou_{B_2}$ does not change the projection of the
support of a  distribution to $W_1$,
$$p_1(\mathrm{Supp}(\Fou_{B_1}(\xi)))=p_1(\mathrm{Supp}(\Fou_{B_2} \circ
\Fou_{B_1}(\xi)))=p_1(\Fou_{B_1 \oplus B_2}(\xi)) \subset Z$$
\end{proof}

For the next theorem, Let  $q:Z \to T$ be a continuous map of $l$-spaces. We can consider $\Sc ^*(Z)$ as $\Sc(T)$-module. Denote $Z_t:=
q^{-1}(t)$.
\begin{theorem}[Localization principle, see \cite{Ber}, section 1.4] \label{LocPrin}
For any $M$ which is a closed linear subspace and $\Sc(T)$-submodule of $\Sc^*(Z)$,
$$M=\overline{\bigoplus_{t \in T} (M \cap \Sc^*(Z_t))}$$.
\end{theorem}
Informally, it means that in order to prove a certain property of
distributions on $Z$ it is enough to prove that distributions on
every fiber $Z_t$ have this property.
\begin{corollary} \label{LocPrinCor}
Let $q:Z \to T$ be a continuous map of $l$-spaces. Let an $l$-group
$H$ act on an $l$-space $Z$ preserving the fibers of $q$. Let $\mu$
be a character of $H$. Suppose that for any $t\in T$,
$\Sc^*(q^{-1}(t))^{H,\mu}=0$. Then $\Sc^*(Z)^{H,\mu}=0$
\end{corollary}

\begin{corollary} \label{Product}
Let $H_i \subset \widetilde{H}_i$ be $l$-groups acting on $l$-spaces
$Z_i$ for $i=1, \ldots, k$. Suppose that
$\Sc^*(Z_i)^{H_i}=\Sc^*(Z_i)^{\widetilde{H}_i}$ for all $i$. Then
$\Sc^*(\prod Z_i)^{\prod H_i}=\Sc^*(\prod Z_i)^{\prod
\widetilde{H}_i}$.
\end{corollary}

For the next theorem, let $H$ be a unimodular $l$-group acting on two $l$-spaces $E$ and $Z$, with the action on $Z$ being transitive. Suppose that we have an $H$-equivariant map $\varphi:E \to Z$. Let $x\in Z$ be a point with a unimodular stabilizer in $H$. Denote by $F$ the fiber of $x$ with respect to $\varphi$.
Then for any character $\mu$ of $H$ the following theorem holds (see \cite{Ber}, section 1.5):
\begin{theorem}[Frobenious descent] \label{Frob} ~\\
%
(i) There exists a canonical isomorphism $\mathrm{Fr}:
\Sc^*(E)^{H,\mu}
\to \Sc^*(F)^{\mathrm{Stab}_H(x),\mu}$.\\
(ii) For any distribution $\xi \in \Sc^*(E)^{H,\mu}$,
$\mathrm{Supp}(\mathrm{Fr}(\xi))=\mathrm{Supp}(\xi)\cap F$.\\
(iii) Frobenious descent commutes with Fourier transform.
\end{theorem}
To formulate (iii) explicitly, let $W$ be a finite dimensional linear space over $\F$ with
a nondegenerate bilinear form $B$, and suppose $H$ acts on $W$ linearly
preserving $B$.
Then for any $\xi \in \Sc^*(Z\times W)^{H,\mu}$, we have
$\Fou_{B}(\mathrm{Fr}(\xi))=\mathrm{Fr}(\Fou_{B}(\xi))$, where $\mathrm{Fr}$ is taken with respect to the projection $Z \times W \to Z$.
\begin{definition} \label{abs-hom}
Let $W$ be a finite dimensional vector space over $\F$. Given a
distribution $\xi \in \Sc^*(W)$ we call it \textbf{abs-homogeneous of degree
$\mathbf{d}$} if for any $f \in \Sc(W)$ and $t\in \R^\times$,
$|\xi(h_{t^{-1}}(f))| = |t|^{-d} |\xi(f)|$ where
$(h_{t^{-1}}(f))(v)= f(tv)$.
\end{definition}
 For example, a Haar measure on $W$ is abs-homogeneous of degree
$\dim W$ and the $\delta$-distribution supported at $0$ is abs-homogeneous of degree
$0$.
\\

A crucial step in the proof of the main theorem is a
special case of a result by Rallis and Schiffmann (this theorem for the case $\mathrm{ch}(\F)=0$ appears in \cite{RS} as lemma 8.1, and the same proof works verbatim for the case of positive characteristic different from 2).

Let $W$ be a finite dimensional vector space over $\F$ and $B$ be a
nondegenerate symmetric bilinear form on $W$. Denote $Z(B):=\{ v \in
W | B(v,v)=0\}$. 

\begin{theorem}[Rallis-Schiffmann] \label{Metaplectic}
Let $\xi$ be a distribution on $W$ such that
both $\xi$ and $\Fou_{B}(\xi)$ are supported in $Z_{B}$.

 Then
$\xi$ is abs-homogeneous of degree $\frac{1}{2} \dim W$.
\end{theorem}
\begin{definition} \label{comp}
A matrix $A \in \mathrm{gl}(n,\F)$ is said to be a companion matrix, if:
\begin{enumerate}
    \item It has $1$ in the entries of the diagonal immediately below the main diagonal.
    \item For all $1\leq i\leq n$ we have $A_{i,n}=-a_{i-1}$, where $f(x)=x^n+a_{n-1}x^{n-1}+\dots+a_0$ is a power of an irreducible polynomial.
    \item All of the other entries of $A$ are $0$.
\end{enumerate}
We also call such a matrix the companion matrix of $f$.
For example, the companion matrix of $f(x)=x^3-3x^2+3x-1$ is $$ \left( \begin{matrix}
    0 & 0 & 1 \\
    1 & 0 & -3 \\
    0 & 1 & 3
  \end{matrix} 
  \right)$$
Both the characteristic polynomial and the minimal polynomial of the companion matrix of $f(x)$ are equal to $f(x)$.
\end{definition}
\begin{theorem} [Rational Canonical Form] \label{Rational}
Any matrix $A\in \mathrm{gl}(n,\F)$ is conjugate to a direct sum of companion matrices. Moreover, this form is unique up to a permutation of the blocks.
\end{theorem}
This form is called the rational canonical form of $A$.

\begin{remark}
Let $Z$ be an $l$-space and $Q\subset Z$ be a closed subset. We will
identify $\Sc^*(Q)$ with the space of all distributions on $Z$
supported on $Q$. In particular, we can restrict a distribution
$\xi$ to any open subset of the support of $\xi$.
\end{remark}

\section{Reformulations of the Problem} \label{Reform}
Let $G:=G_n :=\mathrm{GL}(n,\F)$. Consider the action of the
2-element group $S_2$ on $G$ given by the involution $g \mapsto {\,
(g^{-1})^t}$. It defines a semidirect product which we denote by $\tG := \tG_n := G_n
\rtimes S_2$. Let $V:=V_n:=\F^n$ and $X := X_n := \mathrm{gl}(V)
\times V \times V^*$.

The group $\tG$ acts on $X$ by
$$(g,1).(A,v,\phi):= (gAg^{-1},gv, (g^{-1})^* \phi)$$
$$(g,-1).(A,v,\phi):=(gA^tg^{-1},g\phi^t,(g^{-1})^*v^t)$$
where $g \in G$ and $-1$ is considered as the generator of $S_2$. Here, $A^t$
denotes the transposed  matrix in $\mathrm{gl}_n$, $\phi^t \in V$ denotes
the column vector corresponding to the row vector $\phi \in V^*$,
and $v^t$ denotes the row vector corresponding to the column
vector $v \in V$. Also for any operator $g:V \to V$, we denote by
$g^*:V^* \to V^*$ the adjoint operator.

Note that $\tG$ acts separately on $\mathrm{gl}(V)$ and on $V \times
V^*$. Define a character $\chi$ of $\tG$ by $\chi(g,s):=
\mathrm{sign}(s)$.

In this section we show that theorem \ref{goal} can be deduced from the following theorem.
\begin{theorem} [Main Theorem] \label{main}
Any $(\tG,\chi)$-equivariant distribution on $X$ is zero.
\end{theorem}

\begin{proposition} \label{Red1}
If  $\Sc^*(G_{n+1})^{\tG_n,\chi}=0$ then theorem \ref{goal} holds.
\end{proposition}
\begin{proof}
Let $\xi \in \Sc^*(G_{n+1})^{G_n}$. 
Consider $\eta :=\xi -\xi^t$. Clearly $\eta \in
\Sc^*(G_{n+1})^{\tG_n,\chi}$, hence $\eta=0$ and so $\xi
=\xi^t$.
\end{proof}
\begin{proposition} \label{Red2}
If  $\Sc^*(\mathrm{gl}_{n+1})^{\tG_n,\chi}=0$ then
$\Sc^*(G_{n+1})^{\tG_n,\chi}=0$.
\end{proposition}
\begin{proof}
Let $\xi \in \Sc^*(G_{n+1})^{\tG_n,\chi}$. We have to prove $\xi=0$. Assume the contrary. Take $p \in \mathrm{Supp}(\xi)$. Let $t=\mathrm{det}(p)$.
Let $f\in \Sc(\F)$ be such that $f(0)=0$ and $f(t) \neq 0$. Consider
the determinant map $\mathrm{det}:G_{n+1} \to \F$. Consider
$\xi':=(f \circ \mathrm{det})\cdot \xi$. It is easy to check that
$\xi' \in \Sc^*(G_{n+1})^{\tG_n,\chi}$ and $p \in
\mathrm{Supp}(\xi')$. However, we can extend $\xi'$ by zero to
$\xi'' \in \Sc^*(\mathrm{gl}_{n+1})^{\tG_n,\chi}$, which is zero by
the assumption. Hence $\xi'$ is also zero. This yields contradiction.
\end{proof}


\begin{proposition} \label{Red3}
If $\Sc^*(X)^{\tG_n,\chi}=0$ then
$\Sc^*(\mathrm{gl}_{n+1})^{\tG_n,\chi}=0$.
\end{proposition}
\begin{proof}
Consider the $\tG$-invariant map $q:\mathrm{gl}_{n+1} \to \F$ given
by $q(B):= B_{n+1,n+1}$. By the Localization principle (corollary
\ref{LocPrinCor}), it is enough to prove that for any $t\in \F$,
$\Sc^*(q^{-1}(t))^{\tG_n,\chi}=0$. However, all $q^{-1}(t)$ are
isomorphic as $\tG_n$-equivariant $l$-spaces to $X$ by
$$ \left(
  \begin{array}{cc}
    A_{n \times n} & v_{n\times 1} \\
    \phi_{1\times n} & \lambda \\
  \end{array}
\right) \mapsto (A, v,\phi) $$
\end{proof}

\section{Proof of the Main Theorem}
We prove the main theorem (Theorem \ref{main}) by induction on $n$. That is, we assume that $\Sc^*(X_m)^{\tG_m,\chi}=0$ for all $m<n$, for $\F$ along with all of its finite extensions.

\begin{notation}
Set $\Delta:X\to\F[x]$ to be the map $(A,v,\phi)\mapsto \mathrm{ch}(A)$. This is a continuous map of $\ell$-spaces.
\end{notation}
\subsection{Restriction of the Possible Support 
}

\begin{proposition} \label{perp}
Any $(\tG,\chi)$-equivariant distribution on $X$ is supported on $\{(A,v,\phi)\in X|\phi v=0\}$.
\end{proposition}

\begin{proof}
The map $$\kappa:X\to \F,\ (A,v,\phi)\mapsto \phi v$$ is $\tG$-invariant, and so by Localization principle (theorem \ref{LocPrin}) it is enough to consider $(\tG,\chi)$-equivariant distributions on a single fiber $\kappa^{-1}(a)$ where $a\in \F^\times$, and show they must be $0$.
\\Let $e_n$ denote the last element of the standard basis of $V$, and $e_n^*$ denote the last element of the standard dual basis of $V^*$.
We can use Frobenius descent (\ref{Frob}) on $pr:\kappa^{-1}(a)\to V\times V^*$, as the centralizer of $(ae_n,e_n^*)$ is $\tG_{n-1}$, which is unimodular.\\
This gives us $$\Sc^*(\kappa^{-1}(a))^{\tG_{n}}=\Sc^*(gl_{n-1})^{\tG_{n-1}}$$

We are are left with proving that any $(\tG_{n-1},\chi_{n-1})$-equivariant distribution on $\mathrm{gl}_{n-1}$ is 0. This follows form the main theorem (Theorem \ref{main}) for $n-1$ by Proposition \ref{Red3}, and we assume it in the induction hypothesis.
\end{proof}

\subsection{Introducing Automorphisms of the Problem}

\begin{notation} Consider the following two families of homeomorphisms:
\begin{itemize}
    \item For $\lambda \in \F$, let $\nu_\lambda:X\to X,$ be the homeomorphism defined by $$\nu_\lambda (A,v,\phi)=(A+\lambda v\otimes \phi,v,\phi)$$
    \item Let $f\in \F (x)^\times$, and fix a fiber $F$ of $\Delta$, at a characteristic polynomial coprime to $f$ (both to the numerator and to the denominator).
    Let $\rho _f:F\to F$ be the homeomorphism defined by $$\rho _f (A,v,\phi)=(A,f(A)v,\phi f(A))$$
\end{itemize}
\end{notation}

\begin{remark}
These maps are indeed homeomorphisms, as $\nu_\lambda$ and $\nu_{-\lambda}$ are inverse to each other, so are $\rho_f$ and $\rho_{f^{-1}}$, and all of these maps are defined in a continuous way. Moreover, as  one can check, the maps $\nu_\lambda,\rho_f$ both commute with the action of $\tG$.
\end{remark}

These constructions allow us to amplify the restriction on the support we get from Proposition \ref{perp} into a stronger condition.

\begin{notation} \label{notR}
Set $R:=\{(A,v,\phi)\in \mathrm{gl}(V)\times V\times V^* |\forall k \geq 0, \phi A^k v=0\}$
\end{notation}
\begin{proposition} \label{R}
Let $\xi$ be a $(\tG,\chi)$-equivariant distribution on $X$. Then $\xi$ is supported on $R$.
\end{proposition}
\begin{proof}

To prove the claim it is enough to show that there are no $(\tG,\chi)$-equivariant distributions on $X\setminus R$. By the Localization principle (\ref{LocPrin}), it is enough to show there are no $(\tG,\chi)$-equivariant distributions on $F\cap R^c$ for any fiber $F$ of $\Delta$ (Note that $R$ is $\tG$-invariant).
Let $\xi$ be such a distribution, and let $(A,v,\phi)$ be a point in $\mathrm{supp}(\xi)$.
Let $f\in \F[x]$ be a polynomial coprime to the characteristic polynomial of $A$.
Applying Proposition \ref{perp} to $\rho_f(A,v,\phi)=(A,f(A)v,\phi f(A))$ (after extension to $X$), we get that $\phi f(A)^2 v=0$.
\\Since the set of polynomials relatively prime to the characteristic polynomial of $A$ is Zarisky dense in $\F[x]$, we have that for any $f\in \F[x]$, $\phi f(A)^2 v=0$ holds. In particular for any $k\geq 0$
$$\phi A^k v= \frac{\phi(1+A^k)^2v-\phi(A^k)^2v-\phi v}{2}=0$$
Since we assumed that our distribution is supported on $R^c$, we get that it must be equal to $0$, as its support is empty.
\end{proof}

The importance of the following theorem, which will be proved in section \ref{LinAlgPf}, is already evident.

\begin{theorem} \label{LinAlg}
Let $V$ be a linear space over $\F$ of finite dimension $n$, $A\in \mathrm{gl} (V)$, $v\in V$, and $\phi \in V^*$. The following are equivalent.
\begin{enumerate}
    \item $\forall k\geq 0,\ \phi A^k v=0$.
    \item For all $\lambda\in\F$, $\mathrm{ch}(A+\lambda v\otimes \phi)=\mathrm{ch}A$.
    \item There exists $\lambda\in\F^\times$ such that $\mathrm{ch}(A+\lambda v\otimes \phi)=\mathrm{ch}A$.
\end{enumerate}


\end{theorem}

\begin{corollary}
If $\xi$ is a $(\tG,\chi)$-equivariant distribution on one of the fibers $\Delta^{-1}(g)$ of $\Delta$, then $\nu_{\lambda}(i_*\xi)$ is also supported on $\Delta^{-1}(g)$, where $i$ is the inclusion of $\Delta^{-1}(g)$ into $X$. This way, we can regard $\nu_{\lambda}$ as an automorphism of $\Sc^*(\Delta^{-1}(g))^{({\tG},\chi)}$.
\end{corollary}

\subsection{Stratification}
For any $g\in \mathrm{gl}(\F)$, let $Y_g$ be the subspace of $\mathrm{gl}(\F)$ consisting of matrices with characteristic polynomial $g$.
By the Localization principle (\ref{LocPrin}) and the previous theorems, it is enough for theorem \ref{main} to prove that any $(\tG,\chi)$-equivariant distribution on $\Delta^{-1}(g)=Y_g\times V\times V^*$ is $0$, for any $g$ as above.

The strategy now will be to stratify $Y_g$ and restrict stratum by stratum the possible support for a  $(\tG,\chi)$-equivariant distribution (note that $Y_g$ is a union of finitely many $\tG$ orbits).

 \begin{notation}
Denote by $P_i$ the union of all $\tG$-orbits of $Y_g$ of dimension at most $i$, and let $R_i:=R\cap (P_i\times V \times V^*)$ (where $R$ is as in Notation \ref{notR}). Also, for any open $\tG$-orbit $O$ of $P_i$ set $$\tO:=(O\times V \times V^*)\cap\bigcap_{\lambda\in\F}\nu_\lambda^{-1}(R_i)$$
Note that $P_i$ are Zariski closed inside $Y_g$, $P_k=Y_g$ for $k$ big enough, and $P_{-1}=\emptyset$.
 \end{notation}
 We focus ourselves to proving the following claim, which essentially deals with a single orbit of $P_i$:
\begin{claim} \label{orbit}
Let $O$ be an open $\tG$-orbit of $P_i$. Suppose $\xi$ is a $(\tG,\chi)$-equivariant distribution on $\Delta^{-1}(g)$ such that
$$\mathrm{supp}(\xi) \ss \tO$$
and
$$\mathrm{supp}(\Fou_{V \oplus V^*}(\xi)) \ss \tO$$
Then $\xi=0$.
\end{claim}

Using this claim, the main theorem (Theorem \ref{main}) is easily proven, in the following way:

\begin{theorem}
Any $(\tG,\chi)$-equivariant distribution on $X$ is zero.
\end{theorem}

\begin{proof}
We prove by downward induction the following claim - any $(\tG,\chi)$-equivariant distribution on $\Delta^{-1}(g)$ is supported inside $R_i$. This claim for $i$ big enough is exactly proposition \ref{R}, and the claim for $i=-1$ implies the theorem by the Localization principle, (\ref{LocPrin}), as we have already discussed.
For the induction step, take such a distribution $\xi$. As $P_i\setminus P_{i-1}$ is a disjoint union of open orbits, it is enough to show that the restriction of $\xi$ to any $O\times V\times V^*$, where $O$ is an open orbit of $P_i$, is zero. Let $\zeta=\xi|_{O\times V \times V^*}$ be such a restriction.
By the induction hypothesis, we know that  $\mathrm{supp}(\zeta) \ss \tO$ and $\mathrm{supp}(\Fou_{V \oplus V^*}(\zeta)) \ss \tO$ (Fourier transform doesn't change the projection of the support on the $A$-coordinate). Hence by Claim \ref{orbit}, $\zeta=0$.
\end{proof}
\subsection{Handling a single orbit}
In this section we will finish the proof of the main theorem (Theorem \ref{main}) by proving Claim \ref{orbit}. we shall keep all previous notations, unless mentioned otherwise.
\\

\begin{notation}
For $A\in \mathrm{gl}(V)$, set $$Q_A:=\{(v,\phi)\in V\oplus V^*|v\otimes\phi\in [A,\mathrm {gl} (V)]\}$$ and $$R_A:=\{(v,\phi)\in V\oplus V^*|\forall k\geq 0,\ \phi A^k v=0\}=\{(v,\phi)\in V\oplus V^*|(A,v,\phi)\in R\}$$
\end{notation}

\begin{proposition}
If $(A,v,\phi)\in\tO$ then $(v,\phi)\ss Q_A$.
\end{proposition}

\begin{proof}
Consider a point $(A,v,\phi)\in \tO$.
The Zariski tangent space to $O$ at point $A$ is $[A,\mathrm {gl} (V)]$, and since a Zariski neighborhood of $A$ inside the line $A+\lambda v \otimes\phi$ is contained in $O$ (by \ref{LinAlg}), we have that $v\otimes\phi\ss [A,\mathrm {gl} (V)]$.
\end{proof}

We also have the following theorem, which we will prove in section \ref{LinAlgPf}.

\begin{theorem} \label{QA-RA}
$Q_A \ss R_A$.
\end{theorem}

\begin{notation}
Let $A \in \mathrm{gl}(V)$. We denote by
${C_A}$ the stabilizer of $A$ in $G$ and by ${\widetilde{C}_A}$ the
stabilizer of $A$ in $\tG$.
\end{notation}
It is known that $C_A$ is unimodular and hence $\widetilde{C}_A$ is
also unimodular.

Claim \ref{orbit} follows now from Frobenius descent (\ref{Frob}) and the
following proposition.

\begin{proposition} \label{SubKey}
Let $A \in \mathrm{gl}(V)$. Let $\eta \in \Sc^*(V \oplus V^*)^{C_A}$. Suppose that both $\eta$ and
$\Fou_{V\oplus V^*}(\eta)$ are supported in $Q_A$. Then $\eta \in \Sc^*(V \oplus
V^*)^{\widetilde{C}_A}$.
\end{proposition}

We will call an element $A\in \mathrm{gl}(V)$ 'nice' if
the previous proposition holds for $A$. Namely, $A$ is 'nice' if any
distribution $\eta \in \Sc^*(V \oplus V^*)^{C_A}$ such that both
$\eta$ and $\Fou(\eta)$ are supported in $Q_A$ is also
$\widetilde{C}_A$-invariant.

\begin{lemma} \label{DirectSum}
Let $A_1 \in \mathrm{gl}(\F^k)$ and $A_2 \in \mathrm{gl}(\F^l)$ be nice. Then $A_1 \oplus A_2 \in \mathrm{gl}(\F^{k+l})$ is nice.
\end{lemma}

First we prove the following simple lemma.
\begin{lemma}
$Q_{A_1 \oplus A_2} \subset Q_{A_1} \times Q_{A_2}$.
\end{lemma}
\begin{proof}
Let $(v,\phi) \in Q_{A_1 \oplus A_2}$. This means that
$v \otimes \phi = [A_1 \oplus A_2,B]$, for some $B
$.
Let $B = \left(
                                                  \begin{array}{cc}
                                                    B_{11} & B_{12} \\
                                                    B_{21} & B_{22} \\
                                                  \end{array}
                                                \right)$, $v=v_1+v_2$ and $\phi = \phi_1 +\phi_2$  be the
decompositions corresponding to the blocks of $A_1 \oplus A_2$. Then
$v_1 \otimes \phi_1 = [A_1,B_{11}]$ and $v_2 \otimes \phi_2 =
[A_2,B_{22}]$. Hence $(v_1,\phi_1) \in Q_{A_1}$ and $(v_2,\phi_2)
\in Q_{A_2}$.
\end{proof}

\begin{proof}[Proof of lemma \ref{DirectSum}]

Let $A_1\in \mathrm{gl}(V_1),A_2\in \mathrm{gl}(V_2)$ be as in the lemma. Suppose that we have $\eta \in \Sc^*(V_1 \oplus V_1^*\oplus V_2 \oplus V_2^*)^{C_{A_1\oplus A_2}}$, such that both
$\eta$ and $\Fou_{V_1 \oplus V_1^*\oplus V_2 \oplus V_2^*}(\eta)$ are supported in $Q_A$. We want to show that $\eta$ is is also
$\widetilde{C}_{A_1\oplus A_2}$-invariant.
\\
Proposition \ref{2Four} implies that both $\eta$ and $\Fou_{V_1 \oplus V_1^*}(\eta)$ are supported in $Q_{A_1}\times V_2\times V_2^*$. By the Localization principle (theorem \ref{LocPrin}), along with the fact that $A_1$ is nice, we get that $\eta$ is $\widetilde{C}_{A_1}$-invariant. Similarly, it is $\widetilde{C}_{A_2}$-invariant.

Since $\eta$ is also $C_{A_1\oplus A_2}$-invariant by assumption, we get that it is $\widetilde{C}_{A_1\oplus A_2}$-invariant.
\end{proof}



Using the rational canonical form (Theorem \ref{Rational}) and lemma \ref{DirectSum}, Claim (\ref{orbit}) reduces to the following statement:

\begin{claim} \label{RationalBlock}
Let $A \in \mathrm{gl}(V)$ be a companion matrix. Then $A$ is nice.
\end{claim}


Indeed, let $A\in\mathrm{gl}(V)$ be the companion matrix of a polynomial $f^s$, where $s\geq 0$ (we will allow $s=0$ for the degenerate, albeit important, case $\dim V=0$) and $f\in \F[x]$ is some irreducible polynomial. Let $k:=\deg f$.
\\
It can be easily seen that $V\cong \F[x]/(f(x)^s)$ as $\F[x]$ modules, with $x$ acting on $V$ by $A$. 
Thus the different submodules of $V$ are exactly $U_i=f(x)^i V=f(A)^i V$, and they form a descending filtration of $V$. We have also the dual filtration $U^*_i:=U_{s-i}^\perp$ on $V^*$.
\begin{proposition} \label{filtration}
$R_A=\bigcup_{i=0}^{s} U_i\oplus U^*_{s-i}$
\end{proposition}
\begin{proof}
$\bigcup_{i=0}^{s} U_i\oplus U^*_{s-i}\ss R_A$ is trivial.
For the other direction, take $(v,\phi)\in R_A$, and assume $v\in U_i\setminus U_{i-1}$.
We have $\phi P(A)v=0$ for any polynomial $P\in \F[x]$, i.e. $\phi \in \langle v\rangle^\perp=U_i^\perp=U_{s-i}^*$, where $\langle v\rangle$ stands for the submodule generated by $v$ in $V$.
\end{proof}
It is a known fact that any matrix (over any field) is conjugate to its transpose, so we can choose $g\in \mathrm{GL}(V)$ that satisfies $g A^t g^{-1}=A$. Keeping the notations from section \ref{Reform}, this implies $(g,-1).A := g A^t g^{-1}=A$, thus $(g,-1)\in \tC_A$.\\
Let $T:V\to V^*$ be the isomorphism defined by $T(v)=(g^{-1})^*(v^t)=v^t g^{-1}$. One can easily see that $TAT^{-1}=A^*$.
\\
\begin{lemma}
$$U_i^* = T(U_i) = f(A^*)^i V^*$$
\end{lemma}
\begin{proof}
First, $$T(U_i) = T f(A)^i V = f(TAT^{-1})^i T(V) = f(A^*)^i V^*$$
Now, $$T(U_i)=f(A^*)^i V^*\ss U_i^*$$
but also $$\dim T(U_i)=\dim U_i=(s-i)k=\dim U_i^*$$
Therefore $T(U_i)=U_i^*$. 
\end{proof}
\begin{notation}
Define the action of $\F^{\times}$ on $V \oplus V^*$ by $\rho(\lambda)(v,\phi):=(\lambda v, \lambda^{-1}\phi)$. Given $\eta\in\Sc^*(V\oplus V^*)$, denote by $T(\eta)$ the distribution $(g,-1) (\eta)$ we get from the action of $\tG$ on $\Sc^*(V\oplus V^*)$.
\end{notation}
\begin{proposition} \label{finalblow}
Let $\eta \in \Sc^*(V \oplus V^*)^{\F^{\times}}$. Suppose that
$T(\eta)=-\eta$ and that both $\eta$ and $\Fou(\eta)$ are supported
in $R_A$. Then $\eta=0$.
\end{proposition}
\begin{proof}
Let us use induction on $s$, when we fix $f$. Define $V_s:=\F^{ks}$, with $A_s$ being the companion matrix of $f^s$.
\\
For $s=0$, the claim is trivial, as $\eta=T(\eta)=-\eta$.
\\
Assume $s>0$, and consider the restriction of $\eta$ to $(V_s\setminus U_1)\oplus V_s^*$. It must be supported in $V_s\oplus 0$. On this subspace, the action of $\F^{\times}$ is just homothety, and so our restriction is homothety invariant.
\\
However, by Theorem \ref{Metaplectic}, $\eta$ is abs-homogeneous of degree $ks$, and so it follows that $\eta|_{(V_s\setminus U_1)\oplus V_s^*}=0$.
Thus $\mathrm{supp}(\eta)\ss U_1\oplus V_s^*$. Similarly, $\mathrm{supp}(\eta)\ss V_s\oplus U_1^*$, and so $\eta$ is supported in $U_1\oplus U_1^*$.
\\
For the same reasons, $\Fou(\eta)$ is also supported in $U_1\oplus U_1^*$.
\\
Consequently, $\eta$ is invariant to translation by $(U_1\oplus U_1^*)^{\perp}=U_{s-1}\oplus U_{s-1}^*$.
\\
If $s=1$ this implies $\eta=0$. Otherwise, it means that $\eta$ is the pullback of some distribution $\alpha$ on $(U_1\oplus U_1^*)/(U_{s-1}\oplus U_{s-1}^*)$, a space which can be identified with $V_{s-2}\oplus V_{s-2}^*$ via an identification of $U_i/U_{s-1}$ with $U_{i-1}\ss V_{s-2}$ and a corresponding identification of $U_i^*/U_{s-1}^*=(U_{s-i}/U_{s-1})^{\perp}$ with $U_{i-1}^*\ss V_{s-2}^*$.
\\
In this identification $A_s$ is identified with $A_{s-2}$, the natural bilinear forms on the two spaces are identified with each other, and $T$ is identified with (some legitimate choice of) $T:V_{s-2}\to V_{s-2}^*$. Note that the actions of $A_s$ on $V_s$ and of $T$ on $V_s\oplus V_s^*$ indeed induce actions on $U_1/U_{s-1}$ and $(U_1\oplus U_1^*)/(U_{s-1}\oplus U_{s-1}^*)$ respectively. This is because of $A_s$-invariance of the subspaces $U_i$, and because we had $T(U_i)=U_i^*$.
\\
By the induction hypothesis, we have $\alpha=0$.
\end{proof}

\begin{proof} [Proof of claim \ref{RationalBlock}]
Let $\zeta\in \Sc^*(V\oplus V^*)^{C_A}$. Suppose that both $\zeta$ and $\Fou(\zeta)$ are supported in $Q_A$. By Theorem \ref{QA-RA}, this means that in particular they are supported in $R_A$.
Setting $\eta=\zeta-T(\zeta)$, we get that $\eta\in \Sc^*(V\oplus V^*)^{\widetilde{C}_A,\chi}$. Since $R_A$ is preserved under the action of $\widetilde{C}_A$, both $\eta$ and $\Fou_\eta$ are supported in $R_A$.
It can be easily seen that this means that $\eta$ satisfies the conditions of Proposition \ref{finalblow}, thus $\eta=0$ and $\zeta=T(\zeta)$, from which we deduce that $\zeta\in \Sc^*(V\oplus V^*)^{\widetilde{C}_A}$.
\end{proof}

\section{Lemmas in linear algebra} \label{LinAlgPf}
Our aim in this section is to prove theorems \ref{LinAlg} and \ref{QA-RA}.
In all of the following discussion, $\F$ may be an arbitrary field, and we assume we have $A\in \mathrm{gl}(\F^n)$, $v\in \F^n$, and $\phi\in (\F^*)^n$. For any $B\in \mathrm{gl}(\F^n)$ define $c_k(B)$ as the sum of all of its $k\times k$ principal minors ($c_0(B)=1$). We denote the characteristic polynomial of $B$ by $P_B(x)$, and we have $$P_B(x)=\sum_{k=0}^n {(-1)^k c_k(B) x^{n-k}}$$

\begin{theorem} \label{CH}
For any $k\geq 1$, $$c_k(A+v\otimes\phi)=c_k(A)+\sum_{j=0}^{k-1}{(-1)^j c_{k-1-j}(A)\phi A^j v}$$
\end{theorem}

With this theorem, we can prove the desired theorems \ref{LinAlg} and \ref{QA-RA} easily:
By replacing $v$ with $\lambda v$ we get
$$c_k(A+\lambda v\otimes\phi)=c_k(A)+\lambda\sum_{j=0}^{k-1}{(-1)^j c_{k-1-j}(A)\phi A^j v}$$
For theorem \ref{LinAlg}, $(1)\Rightarrow(2)$ is a direct consequence, and $(2)\Rightarrow(3)$ is trivial. For $(3)\Rightarrow(1)$, use induction on $k$ to show $\phi A^k v=0$ for all $k\geq 0$. If the claim is true for all non-negative integers smaller than $k$ (it might be that $k=0$ and so this condition is trivial), then

\begin{align*}
\begin{split}
c_{k+1}(A)=\; & c_{k+1}(A+\lambda v\otimes\phi)=c_{k+1}(A)+\lambda\sum_{j=0}^k{(-1)^j c_{k-j}(A)\phi A^j v}\\
= \; & c_{k+1}(A)+(-1)^k\lambda \phi A^k v
\end{split}
\end{align*}

hence the claim.
\\
\\
For theorem \ref{QA-RA}, recall that $[A,\mathrm{gl}(V)]$ is the Zarisky tangent space to the conjugacy class of $A$ in $\mathrm{gl}(v)$. Thus if $v\otimes\phi\in [A,\mathrm{gl}(V)]$,
$$\frac{\partial}{\partial \lambda}P_{A+\lambda v\otimes\phi}|_{\lambda=0}=0$$
That is
$$\frac{\partial}{\partial \lambda}c_k(A+\lambda v\otimes\phi)|_{\lambda=0}=0$$
But since we see in our formula that $c_k(A+\lambda v\otimes\phi)$ is linear in $\lambda$, this means that $c_k(A+\lambda v\otimes\phi)=c_k(A)$, and so by Theorem \ref{LinAlg} that we just proved, $v\otimes\phi\in R_A$.

\subsection{Proof of theorem \ref{CH}}

It will be easier to us to think of $c_k(B)$ (where $B$ is an arbitrary matrix) as the sum of all 'placements of $k$ castles', i.e. the sum over all choices of $k$ entries of $B$ of which no two are in the same row or column, of the product of these $k$ entries with an appropriate sign. we shall call these terms of the sum from now on '$k$-placements' (each $k$-placement is a product of $k$-entries of the matrix with an appropriate sign).
\\
\\
Consider the sum of all $k$-placements of $A+v\otimes\phi$. We can express it as a sum of three parts:

\begin{itemize}
    \item $k$-placements of $A$
    \item $k$-placements which involve one entry of $v\otimes\phi$ and $k-1$ entries of $A$
    \item $k$-placements that involve two or more entries of $v\otimes\phi$
\end{itemize}
Since $v\otimes\phi$ is of rank 1, the sum of all expressions of the third kind is $0$.
The sum of expressions of the first kind is $c_k(A)$ by definition, and so it remains to understand the sum of expressions of the second kind.
\\
Since it is linear in the entries of $v\otimes\phi$, it is of the form $\phi C_k v$ for some $C_k\in \mathrm{gl}(\F^n)$, in which $(C_k)_{i,j}$ is the coefficient of $\phi_iv_j$, thus it is the sum of all products of $k-1$ entries of $A$ which complete the $j,i$ entry to a $k$-placement (with signs). For convenience, we shall call these '$(k-1)$-semi-placements' that complete $j,i$.
\\
\\
We want to show $C_k=\sum_{j=0}^{k-1}{(-1)^j c_{k-1-j}(A) A^j}$ and we shall do so by induction.
For $k=1$, we have $C_1=I=c_0(A)A^0$, hence the formula holds.
\\
Consider $k>1$, and assume that the claim holds for $k-1$. We want to prove the recursion formula
$$C_k=c_{k-1}(A)I-AC_{k-1}$$
Which will prove the claim. 
\\
\\
Let us focus on the $1,1$ entry of $C_k$. It will be the sum of all $(k-1)$-semi-placements on $A$ that complete $1,1$. In this case, all of these semi-placements will actually be placements, and more exactly, these are exactly the placements which don't intersect the first row (or column. It is equivalent). So we can express it the following way:

\begin{align*}
\begin{split}
(C_k)_{1,1}=\; & \sum{[\text{$(k-1)$-placements completing $1,1$}]}\\
= \; & \sum{[\text{$(k-1)$-placements}]}-\sum{[\text{$(k-1)$-placements intersecting the}}\\
& \text{first row}]\\
= \; & c_{k-1}(A)-\sum_{i=1}^n{A_{1,i}\cdot\sum{[\text{$(k-2)$-semi-placements completing $1,i$}]}}\\
= \; & c_{k-1}(A)-\sum_{i=1}^n{A_{1,i}\cdot (C_{k-1})_{i,1}} = \; c_k(A)-(AC_{k-1})_{1,1}\\
= \; & (c_k(A)I-AC_{k-1})_{1,1}
\end{split}
\end{align*}
However this equality will also be true after conjugation by any $B\in \mathrm{GL}(n,\F)$. Thus we have $C_k=c_{k-1}(A)I-AC_{k-1}$ by:
\begin{proposition} \label{1,1 entry}
If $M\in \mathrm{gl}(n,\F)$, satisfies $(BMB^{-1})_{1,1}=0$ for all $B\in \mathrm{GL}(n,\F)$, then $M=0$.
\end{proposition}
\begin{proof}
Let $v\in \F^n,\phi\in (\F^n)^*$ such that $\phi v\neq 0$. Then there exists $B\in \mathrm{GL}(n,\F)$ such that $Bv=e_1,\phi B^{-1}=\alpha e_1^*$ (for some $\alpha\in \F^\times$), and so 
$$\phi Mv=\phi B^{-1}BMB^{-1}Bv=\alpha(BMB^{-1})_{1,1}=0$$
Consequently, for any $\phi\neq 0$ we know $\Ker (\phi M)$ contains the complement of a proper subspace of $\F^n$, hence $\phi M=0$. This implies that $M=0$.
\end{proof}
\begin{remark}
The above calculation is valid for all diagonal entries of $C_k$. For non-diagonal entries, computation by hand is a bit more technical, but still doable.
\end{remark}
\begin{remark}
Note that we didn't assume in the proof that $k\leq n$, so we can apply it for $k=n+1$ to get
$$0=C_{n+1}=\sum_{j=0}^{n}{(-1)^j c_{n-j}(A) A^j}=(-1)^nP_A(A)$$
which is Cayley Hamilton theorem.
\end{remark}

\end{document}